\documentclass[a4paper,10pt,reqno, english]{amsart}

\usepackage{amsmath,amssymb,amscd,amsthm,amsfonts}
\usepackage{graphicx,subfigure}
\usepackage[pdfencoding=auto]{hyperref}
\usepackage{dsfont}
\usepackage[dvipsnames]{xcolor}
\usepackage[utf8]{inputenc}
\usepackage{microtype}
\usepackage[alphabetic]{amsrefs}
\usepackage[capitalize]{cleveref}
\usepackage[pdfencoding=auto]{hyperref}

\newtheorem{theorem}{Theorem}[section]

\newtheorem{lemma}[theorem]{Lemma}

\newtheorem{conjecture}[theorem]{Conjecture}

\newtheorem{problem}[theorem]{Problem}
\newtheorem{definition}{Definition}

\newcommand{\rr}{\mathds{R}}
\newcommand{\zz}{\mathds{Z}}

\DeclareMathOperator{\hs}{HS}

\DeclareMathOperator{\vol}{vol}

\title{Bisections of mass assignments using flags of affine spaces}
\author{Ilani Axelrod-Freed}
\address{Massachusetts Institute of Technology, Cambridge, MA 02139}
\email{ilani\_af@mit.edu}

\author{Pablo Sober\'on}

\address{Baruch College, City University of New York.  New York, NY 10010}
\email{pablo.soberon-bravo@baruch.cuny.edu}

\thanks{This project was done as part of the 2021 New York Discrete Mathematics REU, funded by NSF grant DMS 2051026.  Sober\'on’s research is supported by NSF grant DMS 2054419 and a PSC-CUNY TRADB52 award.}
\date{}

\begin{document}

\begin{abstract}
We use recent extensions of the Borsuk--Ulam theorem for Stiefel manifolds to generalize the ham sandwich theorem to mass assignments.  A $k$-dimensional mass assignment continuously imposes a measure on each $k$-dimensional affine subspace of $\mathbb{R}^d$.  Given a finite collection of mass assignments of different dimensions, one may ask if there is some sequence of affine subspaces $S_{k-1} \subset S_k \subset \ldots \subset S_{d-1} \subset \mathbb{R}^d$ such that $S_i$ bisects all the mass assignments on $S_{i+1}$ for every $i$. We show it is possible to do so whenever the number of mass assignments of dimensions $(k,\ldots,d)$ is a permutation of $(k,\ldots,d)$.  We extend previous work on mass assignments and the central transversal theorem.  We also study the problem of halving several families of $(d-k)$-dimensional affine spaces of $\mathbb{R}^d$ using a $(k-1)$-dimensional affine subspace contained in some translate of a fixed $k$-dimensional affine space.  For $k=d-1$, there results can be interpreted as dynamic ham sandwich theorems for families of moving points.
\end{abstract}

\maketitle

\section{Introduction}

Mass partition problems study how one can split finite sets of points or measures in Euclidean spaces given some geometric constraints.  The methods required to solve such problems are often topological, as they involve the study of continuous equivariant maps between associated topological spaces.  These problems have motivated significant research in algebraic topology, discrete geometry, and computational geometry \cites{matousek2003using, Zivaljevic2017, RoldanPensado2021}.  The classic example of a mass partition result is the ham sandwich theorem, conjectured by Steinhaus and proved by Banach \cite{Steinhaus1938}.

\begin{theorem}[Ham sandwich theorem]
	Let $d$ be a positive integer and $\mu_1, \ldots, \mu_d$ be finite measures of $\rr^d$ so that every hyperplane has measure zero for each $\mu_i$.  Then, there exists a hyperplane $S_{d-1}$ of $\rr^d$ so that its two closed half-spaces $S_{d-1}^+$ and $S_{d-1}^-$ satisfy
	\[
	\mu_i (S_{d-1}^+) = \mu_i (S_{d-1}^-) \qquad \mbox{for all }i=1,\ldots, d.
	\]
\end{theorem}

The proof of the ham sandwich theorem uses the Borsuk--Ulam theorem since the space of partitions can be parametrized with a $d$-dimensional sphere.  As we modify the mass partition problems, we may need to study maps between more elaborate topological spaces, such as direct products of spheres \cites{ManiLevitska2006, Blagojevic:2018jc, Hubard2020}, hyperplane arrangements \cites{Blagojevic2018, Blagojevic:2019hh}, and configuration spaces of $k$ points in $\rr^d$ \cites{Soberon:2012kp, Karasev:2014gi, Blagojevic:2014ey, Akopyan:2018tr} among others.  The topological techniques needed to solve such problems range from simple Borsuk--Ulam type theorems to obstruction theory, explicit computation of characteristic classes in cohomology, or the use of topological invariants such as the Fadell-Husseini index \cite{Fadell:1988tm}.  The results in this manuscript use affine Grassmanians and Stiefel manifolds, yet the proof methods rely on simple topological tools.

Several results have been proven recently regarding mass assignments.  Intuitively, a mass assignment on the $k$-dimensional affine subspaces of $\rr^d$ is a way to continuously assign a measure on each $k$-dimensional affine subspaces of $\rr^d$.

\begin{definition}
  Let $k \le d$ be positive integers and $A_k(\rr^d)$ be the set of $k$-dimensional affine subspaces of $\rr^d$.  For $V \in A_{k}(R^d)$,  let $M_{k}(V)$ be the space of finite measures on $V$ absolutely continuous with respect to the Lebesgue measure in $V$, equipped with the weak topology. Let $E=\{(V,\mu^V): V \in A_k(\rr^d), \ \mu^V \in M_{k}(V) \}$. Consider fibre bundle induced by the function $ g:E \to A_{k}(\rr^d) $ such that $g(V,\mu^V) = V$. 
  
  A $k$-dimensional mass assignment $\mu$ is a section of this fibre bundle.  Note that a $d$-dimensional mass assignment in $\rr^d$ is simply a finite measure absolutely continuous with respect to the Lebesgue measure.  Given a $k$-dimensional mass assignment $\mu$ and a $k$-dimensional affine subspace $V$ of $\rr^d$, we denote by $\mu^V$ the measure that $\mu$ induces on $V$.
\end{definition}

 Simple examples of mass assignments include the projection of a fixed measure onto $k$-dimensional affine subspaces or the $k$-dimensional volume of the intersection of $k$-dimensional spaces with a fixed object in $\rr^d$.  The use of mass assignments in mass partition problems was started by Schnider, motivated by a problem of splitting families of lines in $\rr^3$ \cites{Schnider:2020kk, Pilz2021}.

\begin{figure}
    \centering
    \includegraphics[width=0.7\textwidth]{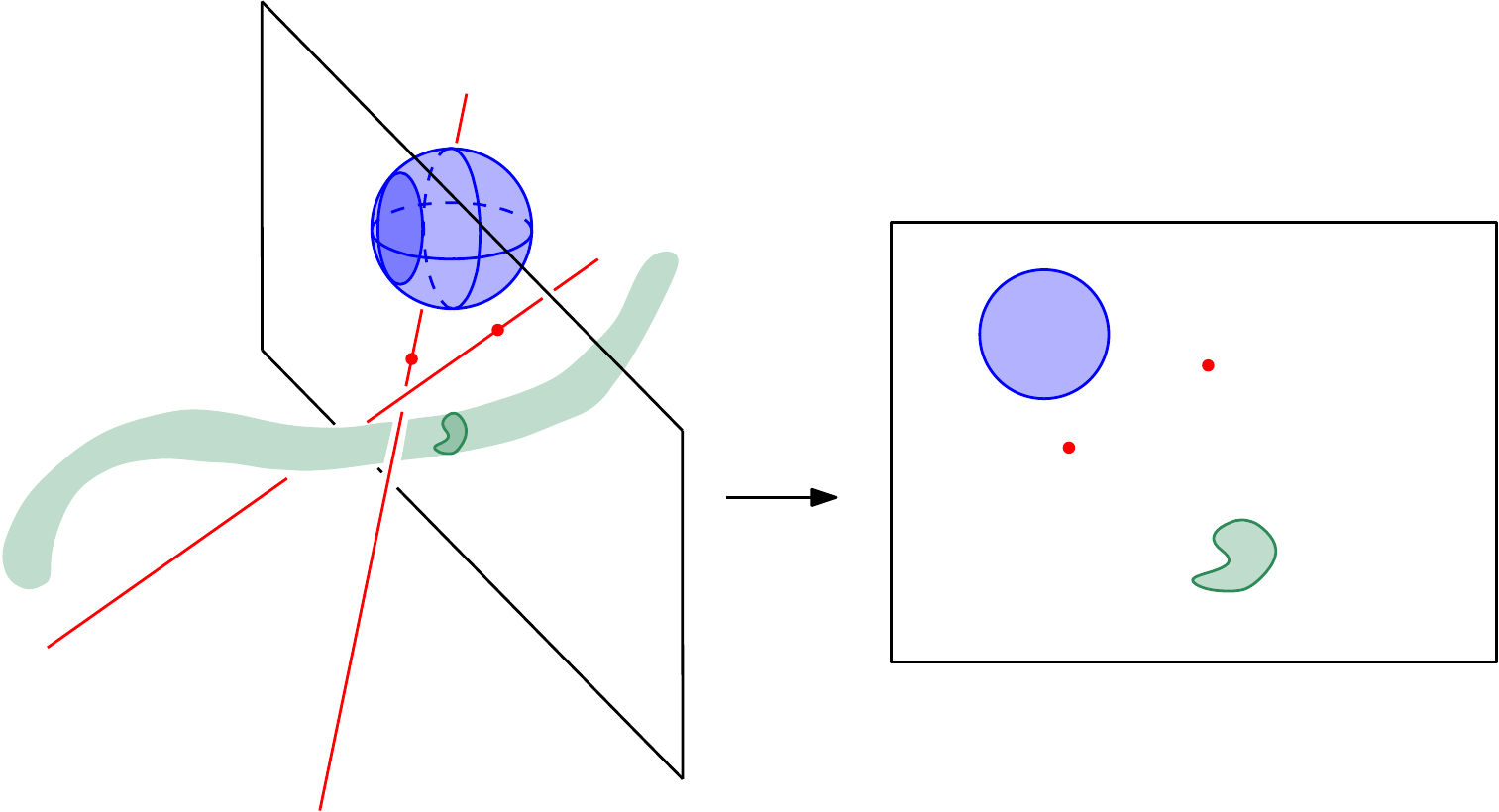}
    \caption{An example of three mass assignments on planes in $\rr^3$, given by intersections of planes with different objects.  The assignment induced by two lines fails the absolute continuity requirement, but provides good intuition.  We are interested in finding a plane where the three induced measures may be simultaneously halved by a line.}
    \label{fig:intro}
\end{figure}

In this paper, we are concerned with splitting mass assignments of several dimensions simultaneously, which we do with full flags of affine subspaces.  A full flag of affine subspaces of $\rr^d$ is a sequence $S_{k-1} \subset S_{k} \subset \ldots \subset S_{d-1} \subset S_d = \rr^d$ where $S_i$ is an $i$-dimensional affine subspace of $\rr^d$ of each $i$ for each $k-1 \le i \le d-1$.  Given several families of mass assignments  of different dimensions, we are interested in finding a full flag where $S_i$ halves each mass assignment on $S_{i+1}$, for $i=k-1,\ldots, d-1$.

Suppose that for $i=k,k+1,\ldots, d$ we have $i$ mass assignments defined on the $i$-dimensional affine subspaces of $\rr^d$.  In particular, we have $d$ measures in $\rr^d$. We can apply the ham sandwich theorem to find a hyperplane $S_{d-1}$ splitting the $d$ measures. Each $(d-1)$-dimensional mass assignment gives us a measure in $S_{d-1}$.  We can apply the ham sandwich theorem again to find $S_{d-2}$, a hyperplane of $S_{d-1}$ which halves each of those $d-1$ measures. Repeated applications of the ham sandwich theorem in this fashion allow us to find a full flag of affine subspaces that simultaneously bisect all the mass assignments. Our main theorem shows that we can permute the number of mass assignments of different dimensions and guarantee the existence of a partition as above.

We state our result on a slightly more general framework where we use arbitrary continuous functions evaluated on half-spaces instead of measures.  Such extensions of mass partition problems have been used recently for the Nandakumar--Ramana-Rao problem \cites{Karasev:2014gi, Blagojevic:2014ey, Akopyan:2018tr}.

\begin{definition}
Let $k \le d$ be non-negative integers and $\hs_k(\rr^d)$ be the space of closed half-spaces of $k$-dimensional affine subspaces of $\rr^d$. An assignment of continuous functions to $k$-dimensional half-spaces is a continuous function $f:\hs_k(\rr^d) \rightarrow \rr$.
\end{definition}

Let $S_{k-1} \subset S_k$ be affine subspaces of $\rr^d$ of dimensions $(k-1)$ and $k$, respectively, and $f$ be an assignment of continuous functions to $k$-dimensional half-spaces.  We say that $S_{k-1}$ bisects $f$ on $S_k$ if, for the two separate closed half-spaces $S_{k-1}^+$ and $S_{k-1}^-$ of $S_k$ on either side of $S_{k-1}$, we have that $f(S_{k-1}^+)=f(S_{k-1}^-)$.  Every $k$-dimensional mass assignment $\mu$ is a particular assignment of continuous functions to $k$-dimensional half-spaces in $\rr^d$.

The following is our main theorem. We nickname it the Fairy Bread Sandwich theorem because one interpretation is splitting a sandwich with multiple base ingredients and multiple colors of sprinkles. Once we make a cut that halves all the base ingredients, we care about the location of the sprinkles in the planar cut.

\begin{theorem}[Fairy Bread Sandwich Theorem]
\label{fairy bread sandwich}
Let $k, d$ be non-negative integers such that $k \le d$ and $\pi=(\pi_{d-1}, \ldots, \pi_{k-1})$ be a permutation of $(d-1,\ldots, k-1)$.  For $i=k-1,\ldots, d-1$, let $\mu_i$ be an $(i+1)$-dimensional mass assignment in $\rr^d$ and $\mathcal{F}_i$ be a set of $\pi_i$ assignments of continuous functions on $(i+1)$-dimensional half-spaces.
There exist affine subspaces $S_{k-1}\subset\ldots\subset S_{d-1} \subset S_d = \rr^d$ such that $S_i$ is an $i$-dimensional affine space of $S_{i+1}$ that bisects $\mu_i$ and all the functions in $\mathcal{F}_i$.

\end{theorem}

The case $k=d$ is the ham sandwich theorem.  The proof of \cref{fairy bread sandwich} uses a Borsuk--Ulam type theorem for Stiefel manifolds, fully described in \cref{sec:main-proof}.  We use \cref{fairy bread sandwich} to strengthen results about mass assignments and to confirm conjectures on the subject.  Our proof method significantly reduces the need for topological machinery to tackle these problems.

\cref{fairy bread sandwich} is optimal for the trivial permutation $\pi = (d-1,\ldots, k-1)$, as the inductive argument shows.  We can choose the mass assignments such that, at each step, the ham sandwich cut is unique, so no more than $i$ mass assignments of dimension $i$ can be cut at any step.  In general, the total number of functions we are halving matches the dimension of the space of full flags we use. We conjecture \cref{fairy bread sandwich} is optimal for all permutations.

If we consider  $\pi=(d-2,d-3,\ldots,k-1,d-1)$ and only use mass assignments, we are halving $d$ mass assignments on the $k$-dimensional space $S_k$, significantly exceeding what a direct application of the ham sandwich theorem gives.  We show that we can impose additional geometric conditions on $S_k$, related to the directions it contains.

\begin{definition}
Let $k \le d$ be positive integers and $e_1, \ldots, e_d$ be the canonical basis of $\rr^d$.  Given an affine subspace $S$ of $\rr^d$, we say that $S$ is $k$-vertical if it contains rays in the directions $e_{d-k+1}, \ldots, e_{d-1}, e_{d}$.
\end{definition}

The following theorem confirms a conjecture by Schnider \cite{Schnider:2020kk}*{Conj.2.4}, which was previously known for $d-k+2$ mass assignments instead of $d$.

\begin{theorem}\label{thm:rotation-case}
Let $k \le d$ be positive integers.  Given a $k$-dimensional mass assignment $\mu$ and $d-1$ assignments of continuous functions on $k$-dimensional half-spaces $f_1,\ldots,f_{d-1}$, there exists a $(k-1)$-vertical $k$-dimensional linear subspace $S_k$ and a $(k-1)$-dimensional affine subspace $S_{k-1} \subset S_k$ such that $S_{k-1}$ bisects $\mu^{S_k}$ and all the functions $f_1,\ldots,f_{d-1}$ on $S_k$.
\end{theorem}

Another application of \cref{fairy bread sandwich} is the existence of central transversals for mass assignments.  The central transversal theorem is a gem of discrete geometry.  It guarantees the existence of affine subspaces which are ``very deep'' within a family of measures, due to Dol'nikov and to \v{Z}ivaljevi\'c and Vre\'cica \cites{Dolnikov:1992ut, Zivaljevic1990}.  For a finite measure $\mu$ in $\rr^d$, we say that an affine subspace $L$ is a $\lambda$-transversal if the dimension of $L$ is $\lambda$ and every closed half-space that contains $L$ has measure greater than or equal to $\left(\frac{1}{d+1-\lambda}\right)\mu(\rr^d)$.

\begin{theorem}[Central transversal theorem]
    Let $\lambda < d$ be non-negative integers.  Let $\mu_1, \ldots, \mu_{\lambda+1}$ be finite measures in $\rr^d$.  Then, there exists an affine subspace $L$ of $\rr^d$ that is a $\lambda$-transversal for each of $\mu_1, \ldots, \mu_{\lambda+1}$.
\end{theorem}

The case $\lambda = 0$ is Rado's centerpoint theorem \cite{Rado:1946ud}, and the case $\lambda = d-1$ is the ham sandwich theorem.  A simple proof of the central transversal theorem was recently found by Manta and Sober\'on \cite{Manta2021}.  We strengthen the results of Schnider for central transversals for mass assignments \cite{Schnider:2020kk}*{Thm 4.3} in a similar way as \cref{thm:rotation-case}.
 
\begin{theorem}\label{thm:center-transversal-vertical}
	Let $0 \le \lambda < k \le d$ be non-negative integers.  Let $\mu_1, \ldots, \mu_{d-k+\lambda+1}$ be $k$-dimensional mass assignments in $\rr^d$.  Then there exists a $\lambda$-vertical, $k$-dimensional linear subspace $S_k$ of $\rr^d$ and an affine subspace $L \subset S_k$ of dimension $\lambda$ such that $L$ is a $\lambda$-transversal to each of $\mu_1^{S_k}, \ldots, \mu_{d-k+\lambda+1}^{S_k}$.
\end{theorem}

In the theorem above, note that since $L$ is a $\lambda$-transversal in a $k$-dimensional space, we require that each closed half-space of $S_k$ that contains $L$ has a $({1}/(k-\lambda+1))$-fraction of each $\mu_i^{S_k}$.  The case $k=d$ is the central transversal theorem.  The case $\lambda=k-1$ is \cref{thm:rotation-case} when all functions are induced by mass assignments.
 
 \cref{thm:rotation-case} guarantees that, given $d$ mass assignment on subspaces of dimension $k$, we can find a $k$-dimensional subspace $S_k$ for which $k-1$ directions are fixed in which a hyperplane halves all measures.  We can think of the last direction of $S_k$ to be allowed to rotate in $S^{d-k}$, a $(d-k)$-dimensional sphere.  If, instead of rotations, we look for translations, we obtain exciting problems.  In this scenario, we fix all directions of $S_k$ but allow it to be translated, so we now have $\rr^{d-k}$ choices for $S_k$.  We need to impose additional conditions for the mass assignments since, otherwise, it's easy to construct mass assignments that are constant in all translated copies of a particular $S_k$.
 
 We are particularly interested in the case of mass assignments induced by families of subspaces.  A set of $(d-k)$-dimensional affine subspaces in $\rr^d$ induces a discrete measure on any $k$-dimensional affine subspace $S_k$.  We simply look for all $(d-k)$-dimensional subspaces $L$ for which $L \cap S_k$ is a point, and consider the resulting subset of $S_k$ as a discrete measure.  In sections \ref{sec:horizontal-separating-spaces} and \ref{sec:dynamic-ham-sandwich} we define continuous versions of these assignments which do induce mass assignments.
 
 We show that these mass assignments are restrictive enough to obtain positive ham sandwich results for the translation case.  The original motivation to work on mass assignments was a conjecture of Barba about bisections of three families of lines in $\rr^3$ with one more line \cite{Schnider:2020kk}.  We show in \cref{thm:horizontal-and-lines} that for $k=d-1$ we can halve $d$ mass assignments of dimension $d-1$ and require that $S_{d-1}$ is $1$-vertical and $S_{d-2}$ is perpendicular to $e_d$.  In terms of Barba's conjecture, it shows that we can halve three families of lines in $\rr^3$ using a horizontal line.
 
 If we fix a line $\ell$ in $\rr^d$ and translate a hyperplane $S_{d-1}$, the intersection point $\ell \cap S_{d-1}$ moves at a constant velocity.  Therefore, the problem of halving mass assignments induced by lines in some translate of a hyperplane is equivalent to halving families of points in $\rr^{d-1}$ that move at constant velocities.  The goal is to show that at some point in time we can simultaneously halve more families of points than the dimension.  We call such results ``dynamic ham sandwich theorems''.  Interestingly, whether this is possible depends on the parity of $d$, as we show in \cref{thm:translation-strange}.  The dynamic ham sandwich results are therefore not a direct consequence of \cref{fairy bread sandwich}.  For the case $k=1$, halving families of hyperplanes in the translate of some fixed line, there is no dependence on the parity of $d$, as we show in \cref{thm:translated-line}.

We can compare our results from \cref{sec:dynamic-ham-sandwich} with earlier results for splitting families of hyperplanes with lines.  The dual theorem to recent developments regarding bisection with hyperplane arrangements \cite{Blagojevic2018} (see \cite{RoldanPensado2021}*{Thm 3.4.7}) shows that \textit{Given $2d - O(\log d)$ families of hyperplanes in $\rr^d$, each with an even cardinality, there exists a segment or an infinite ray that intersect exactly half the hyperplanes in each family}.  Our results show that we can do something similar for $d$ families of hyperplanes if we further fix the direction of the intersecting line and only consider infinite rays instead of segments.

In \cref{sec:main-proof} we prove \cref{fairy bread sandwich} and show how it implies \cref{thm:rotation-case}.  Then, we prove \cref{thm:center-transversal-vertical} in \cref{sec:central-transversal}.  We show in \cref{sec:horizontal-separating-spaces} that results similar to \cref{thm:rotation-case} can be proved if we replace the condition of $S_{k}$ going through the origin by conditions on $S_{k-1}$, when $k=d-1$.  Finally, we present our results on dynamic ham sandwich theorems in \cref{sec:dynamic-ham-sandwich}.   We include open questions and conjectures throughout the paper.

\section{Topological tools and proof of \cref{fairy bread sandwich}}\label{sec:main-proof}

We use a recent generalization of the Borsuk--Ulam theorem for Stiefel manifolds.  Borsuk--Ulam type results for Stiefel manifolds have been studied extensively \cites{Fadell:1988tm, Dzedzej1999, Chan2020}.  For $k \le d$, the Stiefel manifold $V_k(\rr^d)$ is the set of orthonormal $k$-frames in $\rr^d$.  Formally,
\[
V_k(\rr^d) = \{(v_1,\ldots, v_k) \in (\rr^d)^k : v_1,\ldots, v_k \mbox{ are orthogonal unit vectors} \}.
\]
The Stiefel manifold $V_k(\rr^d)$ has a natural free action of $(\zz_2)^k$.  We consider $\zz_2 = \{-1,1\}$ with multiplication.  Then, for $(v_1,\ldots, v_k) \in V_k(\rr^d)$, $(\lambda_1,\ldots, \lambda_k) \in (\zz_2)^k$, and $(x_1,\ldots, x_k) \in \rr^{d-1}\times \rr^{d-2}\times \ldots \times \rr^{d-k}$ we define
\begin{align*}
    (\lambda_1,\ldots, \lambda_k)\cdot (v_1,\ldots, v_k) & = (\lambda_1 v_1, \ldots, \lambda_k v_k) \\
    (\lambda_1,\ldots, \lambda_k)\cdot (x_1,\ldots, x_k) & = (\lambda_1 x_1, \ldots, \lambda_k x_k)
\end{align*}

Our main topological tool is the following theorem.

\begin{theorem}[Chan, Chen, Frick, Hull 2020 \cite{Chan2020}]\label{thm:topo}
    Let $k \le d$ be non-negative integers.  Every continuous $(\zz_2)^k$-equivariant map $f:V_k(\rr^d) \rightarrow \rr^{d-1}\times \ldots \times \rr^{d-k}$ has a zero.
\end{theorem}

For the case $k=d$, we consider $\rr^0 = \{0\}$.  The proof of \cref{thm:topo} uses a topological invariant composed of a sum of degrees of associated maps on spheres.  It is also a consequence of the computations of Fadell and Husseini on $T(k)$-spaces \cite{Fadell:1988tm} and can be proved by elementary homotopy arguments \cites{Mus12, Manta2021}.  \cref{thm:topo} has applications to other mass partition problems \cites{Manta2021, Soberon2021}.  We are now ready to prove \cref{fairy bread sandwich}

\begin{proof}[Proof of \cref{fairy bread sandwich}]
First, let us give an overview of the proof.  We will parametrize the set of possible choices for $S_{k-1},\ldots, S_{d-1}$ using the Stiefel manifold $V_{d-k+1}(\rr^d)$ . Then we construct a continuous $((\zz_2)^{d-k+1})$-equivariant map $F_{\pi}:V_{d-k+1}(\rr^d) \rightarrow \rr^{d-1}\times \ldots \times \rr^{k-1}$ determined by the measures $\mu_{d-1},\ldots,\mu_{k-1}$, the families $\mathcal{F}_{d-1},\ldots,\mathcal{F}_{k-1}$ and the permutation $\pi$. The zeroes of this map correspond to choices for $S_{k-1},\ldots, S_{d-1}$ that split all of the measures and the functions in each family exactly in half, finishing the proof.

Given a permutation $\pi=(\pi_{d-1},\ldots,\pi_{k-1})$ of $(d-1,\ldots,k-1)$ and an element $v=(v_{d-1},\ldots,v_{k-1}) \in V_{d-k+1}(\rr^d)$. We obtain the corresponding $S_{k-1},\ldots, S_{d-1}$ as follows:

\begin{itemize}
    \item First, we define $S_{d} = \rr^d$.
    
    \item For $i\in \{d-1,\ldots, k-1\}$, assume $S_{i+1}$ has been constructed.  We pick $S_i$, the hyperplane of $S_{i+1}$ perpendicular to $v_{\pi_i}$ that splits the measure $\mu_{i}^{S_{i+1}}$ in half.  If there are multiple options for the location of $S_i$ under these rules, we pick $S_i$ to be the midpoint of all possibilities.
\end{itemize}

\begin{figure}
       \begin{center}
       \includegraphics[width=0.9\textwidth]
        {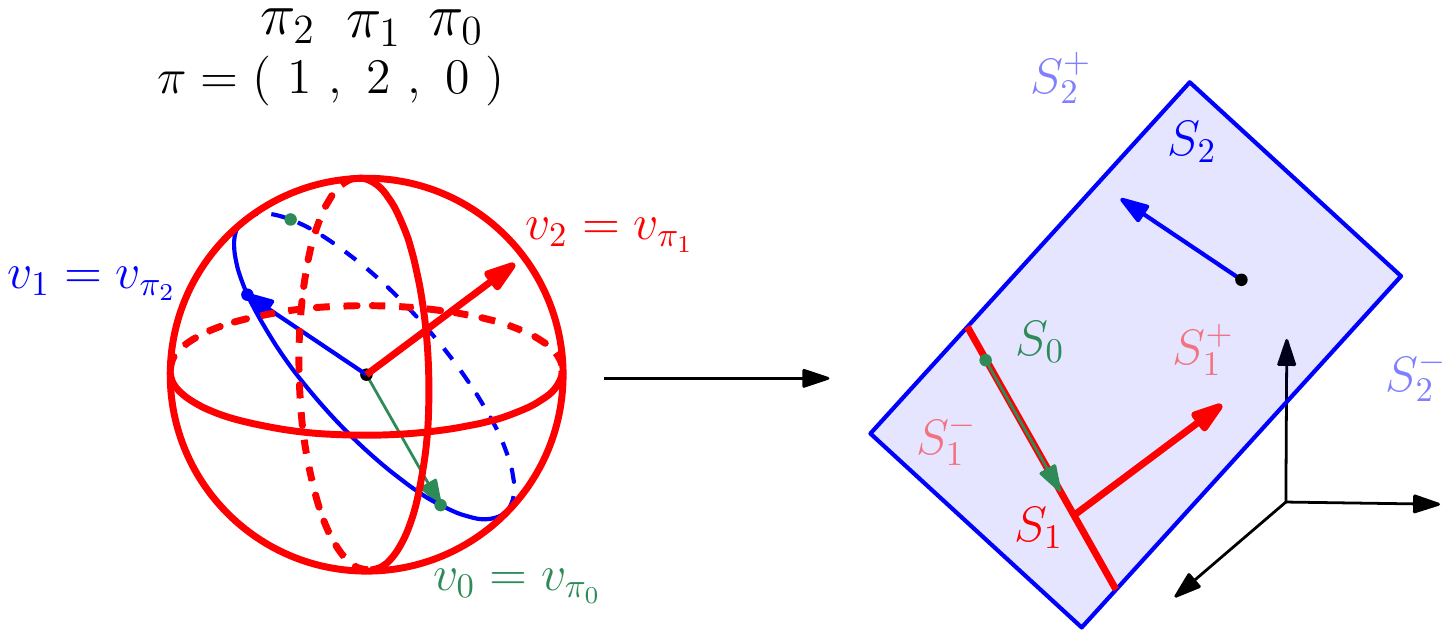}\\[0.3cm]
        $\displaystyle F_{\pi}(v_2,v_1,v_0) = \left( \textcolor{red}{\begin{bmatrix}
        f_{1,1}(S_1^+)-f_{1,1}(S_1^-) \\
        f_{1,2}(S_1^+)-f_{1,2}(S_1^-)
        \end{bmatrix}}, \textcolor{blue}{[f_{2,1}(S_2^+) - f_{2,1}(S_2^-)]}, \textcolor{OliveGreen}{0} \right).$
        \caption{ Example of how $S_2,S_1,S_0$ and value of $F_\pi$ are obtained from the element $(v_2,v_1,v_0) \in V_3(\rr^3)$ and the permutation $\pi=(1,2,0)$.  The permutation tells us which element of $(v_2,v_1,v_0)$ gives us which subspace in the full flag.  
        In this case $S_2$ halves one measure and one function in $\rr^d$, $S_1$ halves one measure and two functions in $S_2$, and $S_0$ only halves one measure in $S_1$.\label{fig:pickingSs}}
        \end{center}
        \end{figure}

Let $S_{i}^+ = \{s_i+av_{\pi_i}:s_i\in S_i, a \geq 0\} \subset S_{i+1}$ and $S_{i}^- = \{s_i-av_{\pi_i}:s_i\in S_i, a \geq 0\}\subset S_{i+1}$
be the complimentary half-spaces of  $S_{i+1}$ bounded by $S_{i}$.

We define the map

\begin{align*}
    F_{\pi} : V_{d-k+1}(\rr^d) & \to \rr^{d-1} \times \ldots \times \rr^{k-1} \\
    (v_{d-1},\ldots, v_{k-1}) & \mapsto (x_{d-1},\ldots, x_{k-1})
\end{align*}

by setting $x_{\pi_i, j}$, the $j$-th coordinate of $x_{\pi_i}$ as

\begin{equation}
    x_{\pi_i,j}=f_{i,j}(S_{i}^+)-f_{i,j}(S_{i}^-).
\end{equation}

We note that $F_{\pi}$ is continuous.  The flag $(S_{k-1},\ldots, S_{d-1})$ does not change if we change the sign of any $v_i$.  If we change the sign of $v_{\pi_i}$, all the half-spaces defined remain the same except for $S_{i}^+$ and $S_{i}^-$, which are flipped.  This causes only $x_{\pi_i}$ to change signs, so $F_{\pi}$ is $((\zz_2)^{d-k+1})$-equivariant.

By \cref{thm:topo}, the map $F_{\pi}$ has a zero.  This zero corresponds to a choice of $S_{k-1} \subset \ldots \subset S_{d-1}$ that also splits all functions in each family in half simultaneously. Note that $x_{\pi_i}$ has $\pi_i$ coordinates, which corresponds to $S_{i}$ bisecting the $\pi_i$ functions in $\mathcal{F}_i$ in $S_{i+1}$.
\end{proof}

\begin{conjecture}
\label{optimality}
Let $k \le d$ be non-negative integers.  For any permutation $\pi$ of $(d-1,\ldots, k-1)$, \cref{fairy bread sandwich} is optimal.
\end{conjecture}

In the specific case where some affine subspace is required to bisect more than $d-1$ assignments of continuous functions, this is quickly apparent. We show examples where all functions are induced by mass assignments.

\begin{lemma}
\label{not over d}
Let $k,d$ be non-negative integers such that $k\leq d$. There exist $d+1$ $k$-dimensional mass assignment $\mu_1,\ldots, \mu_{d+1}$ in $\rr^d$ such that there does not exist any $k$-dimensional affine subspace $S_k$ of $\rr^d$ in which $\mu^{S_k}_1,\ldots, \mu^{S_k}_{d+1}$ can be simultaneously bisected.
\end{lemma}

\begin{proof}
Let $p_1, \ldots, p_{d+1}$ be points in general position in $\rr^d$, so that no hyperplane contains all of them.  For each $p_i$ let $\mu_i$ be the mass assignment induced by the projection of a uniform measure in the unit ball centered at $p_i$ onto a subspace.  Given affine subspaces $S_{k-1} \subset S_k$ of dimensions $k-1$ and $k$, respectively, we know that $S_{k-1}$ halves $\mu_i^{S_k}$ if and only if it contains the projection of $p_i$ onto $S_k$.

The points $p_1, \ldots, p_{d+1}$ affinely span $\rr^d$. Thus, for any affine subspace $S_k$ of $\rr^d$, the projections of $p_1, \ldots, p_{d+1}$ onto $S_k$ affinely span $S_k$. There is then no hyperplane $S_{k-1}$ of $S_k$ containing the projections of all the centers, so $\mu_1^{S_k},\ldots,\mu_{d+1}^{S_k}$ cannot all be simultaneously bisected.
\end{proof}

As a quick application of \cref{fairy bread sandwich}, we prove \cref{thm:rotation-case}.  Recall that an affine subspace of $\rr^d$ is called $k$-vertical if it contains rays in the directions $e_d, e_{d-1}, \ldots, e_{d-k+1}$.

\begin{proof}[Proof of \cref{thm:rotation-case}]

To apply \cref{fairy bread sandwich}, we construct some $(i+1)$-dimensional mass assignments for $i=k,\ldots, d-1$.  For $k\le i < d$, we construct $k$ of these assignments, denoted $\mu_{i,j}$ of $0 \le j < k$.  Let $x_0$ be the origin and for $0 < j < k$ let $x_j=e_{d+1-j}$, the $(d+1-j)$-th element of the canonical base of $\rr^d$. Let $D_{x_j}$ be the unit $d$-dimensional ball centered at $x_j$.

We define $\mu_{i,j}$ by setting $\mu_{i,j}^{S_{i+1}}(A) = \vol_{i+1}(A \cap D_{x_j})$ for any all measurable sets $A \subset S_{i+1}$.  We now induct on $d-i$ to show that each successive $S_i$ for $k \le i < d$ must be $(k-1)$-vertical.  See \cref{fig:rotation-case}.

In the case where $i=d-1$, a hyperplane $S_{d-1}$ bisects the measures $\mu_{d-1,j}$ if and only if it contains $x_j$.  This implies that $S_{d-1}$ goes through the origin and is $(k-1)$-vertical.  Now, assuming $S_{i+1}$ is a $(k-1)$-vertical subspace of dimension $i+1$, the intersection $D_{x_j} \cap S_{i+1}$ is a unit ball of dimension $i+1$ centered at the point $x_j$. Thus a hyperplane $S_{i}$ of $S_{i+1}$ halves $\mu_{i,j}^{S_{i+1}}$ if and only if $x_j \in S_i$. Again, this implies that $S_i$ goes through the origin and is $(k-1)$-vertical.

We apply \cref{fairy bread sandwich} with the permutation $(d-2,d-3,\ldots, k-1, d)$ of $(d-1,d-2,\ldots, k-1)$, using only mass assignments.  Since we have fewer mass assignments than required, there exists a full flag $S_{k-1} \subset \ldots \subset S_{d-1}$ where $S_i$ bisects all measures or functions of $S_{i+1}$ for $i=k-1,\ldots,d-1$.  Since $S_k$ must be a $(k-1)$-vertical subspace of dimension $k$, we obtain the desired conclusion.
\end{proof}

\begin{figure}[h]
       \begin{center}
       \includegraphics[width=.4\textwidth]{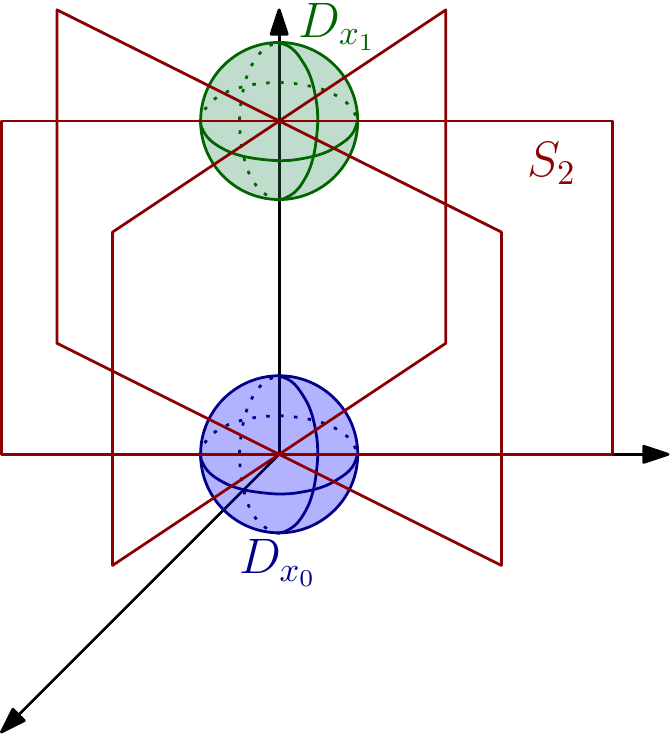}
        \caption{\label{fig:rotation-case} Illustration of the proof of \cref{thm:rotation-case} for $d=3$: Any affine subspace $S_2 \subset \rr^3$ that simultaneously bisect the balls $D_{x_0}$ and $D_{x_1}$ must contain their centers $x_0$ and $x_1$ and therefore the whole $z$ axis.} 
        \end{center}
\end{figure}

Here we used measures and assignments of continuous functions induced by intersections of spheres with $k$-dimensional subspaces to impose conditions on $S_{k},\ldots,S_{d-1}$.  It seems that different geometric constraints could be obtained by modifying the mass assignments.

\begin{problem}
Determine what geometric conditions on $S_k,\ldots, S_{d-1}$ we can impose by choosing other measures and assignment of functions in the proof of \cref{thm:rotation-case}.
\end{problem}

\section{Central transversals to mass assignments}\label{sec:central-transversal}

Schnider proved that if we are given $d+\lambda-k+1$ mass assignments on the $k$-dimensional linear subspaces of $\rr^d$, we can find one $k$-dimensional subspace where all the mass assignments have a common central $\lambda$-transversal \cite{Schnider:2020kk}*{Thm. 4.3}.  His results focus on $k$-dimensional linear subspaces of $\rr^d$, so we also impose that condition in this section to make our results directly comparable.  We extend his result and show that we can also require that the $k$-dimensional linear subspace we find is $\lambda$-vertical.

Recall that given a finite measure $\mu$ in a $k$-dimensional space $S_k$, a central $\lambda$-transversal is an affine $\lambda$-dimensional subspace such that every half-space of $S_k$ that contains the $\lambda$-transversal has measure at least $\mu(S_k)/(k+1-\lambda)$.

\begin{proof}[Proof of \cref{thm:center-transversal-vertical}]
	Let $V_{d-\lambda}(\rr^d)$ be the Stiefel manifold of orthonormal $(d-\lambda)$-frames in $\rr^d$.  We denote $(v_1, \ldots, v_{d-\lambda})$ an element of $V_{d-\lambda}(\rr^d)$.
	
	We denote a few key spaces.  For $i=1,\ldots, d-k$, let
	\[
	S_{d-i} = \left(\operatorname{span}\{v_{k-\lambda+1}, \ldots, v_{k-\lambda+i}\}\right)^{\perp}.
	\]
	Notice that $S_k \subset S_{k+1} \subset \ldots \subset S_{d-1}$, each goes through the origin, and that the dimension of $S_i$ is $i$ for $i=k, \ldots, d-1$.  We also consider $M_{k-\lambda}=\operatorname{span}\{v_1, \ldots, v_{k-\lambda}\}\subset S_k$.
	
	The subspace $S_k$ is our candidate for the $k$-dimensional linear space we look for, and $M_{k-\lambda}$ is our candidate for the orthogonal complement to $L$ in $S_k$.  The subspaces $S_{k+1},\ldots, S_{d-1}$ will be used to guarantee that $S_k$ is $\lambda$-vertical.  To do this, we construct an appropriate $(\zz_2)^{d-\lambda}$-equivariant function
	\begin{align*}
	f:V_{d-\lambda}(\rr^d) & \to \rr^{d-1} \times \rr^{d-2} \times \ldots \times \rr^{\lambda}.	 \\
	(v_1, \ldots, v_{d-\lambda}) &\mapsto (x_1, x_2, \ldots, x_{d-\lambda}).
	\end{align*}
	
	Notice that $x_i \in \rr^{d-i}$ for each $i$.  First, we define the vectors $x_{1}, \ldots, x_{k-\lambda}$.  For $j =1,\ldots, d-k+\lambda+1$, let $\sigma_j$ be the orthogonal projection of $\mu_j^{S_k}$ onto $M_{\lambda-k}$.  We know that the set of centerpoints of $\sigma_j$ is a convex compact subset of $M_{\lambda-k}$, so we can denote by $p_j$ its barycenter.  The affine subspace
	\[
	p_j + (\operatorname{span}\{v_{1}, \ldots, v_{d-\lambda}\})^{\perp} \subset S_k
	\]
	is a $\lambda$-transversal to $\mu_j^{S_k}$, see \cref{fig:center-transversal3d}.
	
	\begin{figure}[h]
       \begin{center}
       \includegraphics[width=\textwidth]
        {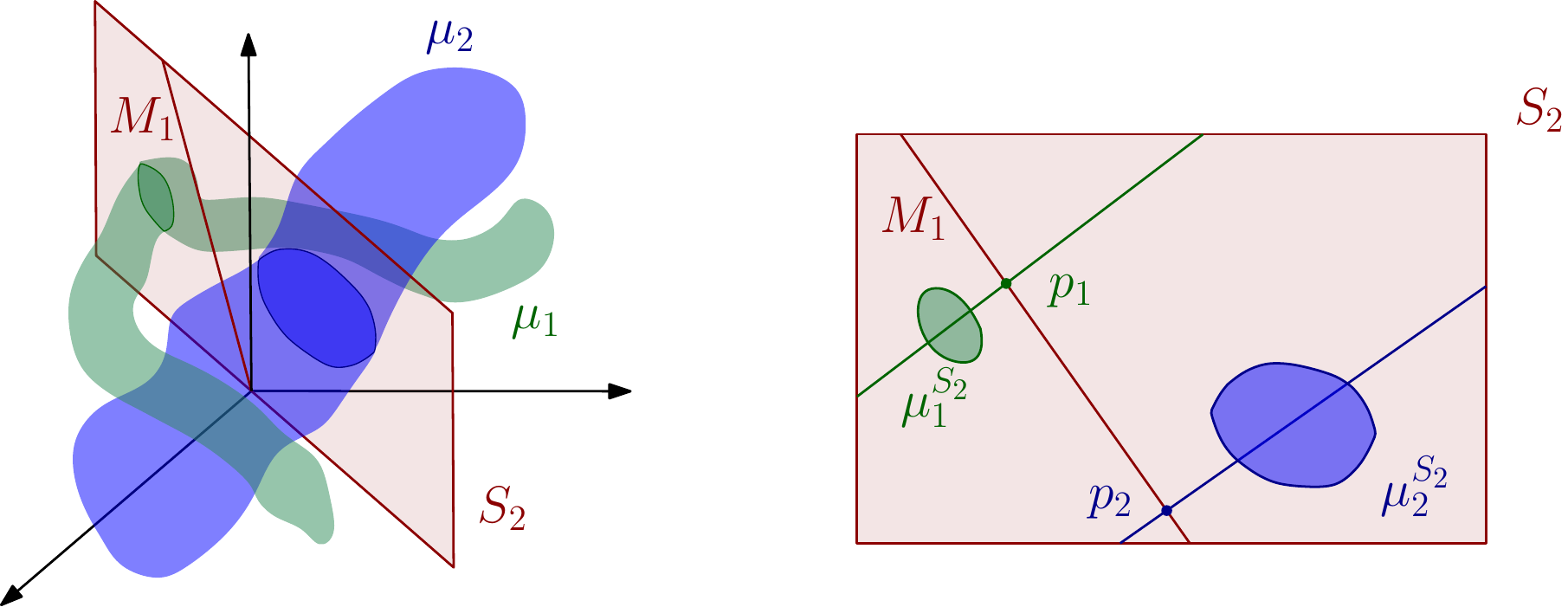}
        \caption{An illustration of the case $(d,k,\lambda) = (3,2,1)$.  The mass assignments $\mu_1, \mu_2$ are induced by the intersection of each plane in $\rr^3$ with the two shapes in the figure.\label{fig:center-transversal3d} } 
        \end{center}
        \end{figure}
	
	We need $p_1, \ldots, p_{d-k+\lambda+1}$ to coincide to guarantee a common $\lambda$-transversal in $S_k$ for all our mass assignments.  For $i=1, \ldots, k-\lambda$ and $j=1,\ldots, d-k+\lambda$, we define the $j$-th coordinate of $x_i$ to be the dot product 
	\[
	\langle v_i, p_j - p_{d-k+\lambda+1}\rangle.
	\]
	The rest of the coordinates of each $x_i$ are zero.  If $x_1, \ldots, x_{k-\lambda}$ are all zero vectors this would imply that for every $j=1,\ldots, d-k+\lambda$ we have
	\begin{align*}
		p_j = \sum_{i=1}^{k-\lambda}\langle v_i, p_j\rangle v_i =  \sum_{i=1}^{k-\lambda}\langle v_i, p_{d-k+\lambda+1}\rangle v_i = p_{d-k+\lambda+1},
	\end{align*}
	as we wanted.  Changing the sign of any of $v_i$ for $i=1,\ldots, k-\lambda$ does not change $S_k$, $M_{k-\lambda}$, nor any of the points $p_j$, and simply flips the sign of $x_i$.
	
	Now let us define $x_{k-\lambda+1}, \ldots, x_{d-\lambda}$.  We denote
	\begin{align*}
		S_{d-i}^+ & = \{y \in S_{d-i+1}: \langle v_{k-\lambda+i}, y \rangle \ge 0 \}, \\
		S_{d-i}^- & = \{y \in S_{d-i+1}: \langle v_{k-\lambda+i}, y \rangle \le 0 \}.
	\end{align*}
	
	For $i=1$ we consider $S_d = \rr^d$.  We also consider, as in the proof of Theorem \ref{thm:rotation-case}, the $\lambda$ mass assignments $\tau_{d-i,1},\ldots, \tau_{d-i,\lambda}$ where $\tau_{d-i,j}$ is a $(d-i)$-dimensional mass assignment such that for any measurable set $C \subset S_{d-i}$ we have
	\[
	\tau^{S_{d-i}}_{d-i,j} (C) = \vol_{d-i}(C \cap D_{e_{d+1-j}}).
	\]
	
	For $i=1, \ldots, d-k$, we define $x_{k-\lambda+i}$ as a vector in $\rr^{d-k+\lambda-i}$ whose first $\lambda$ coordinates are 
	\[
	\tau^{V_{d+1-i}}_{d+1-i,j} (S_{d-i}^+)-\tau^{V_{d+1-i}}_{d+1-i,j} (S_{d-i}^-)
	\]
	 for $j=1,\ldots, \lambda$.  The rest of the coordinates of each $x_i$ are defined to be zero.
	 
	 	\begin{figure}[h]
       \begin{center}
       \includegraphics[width=.7\textwidth]
        {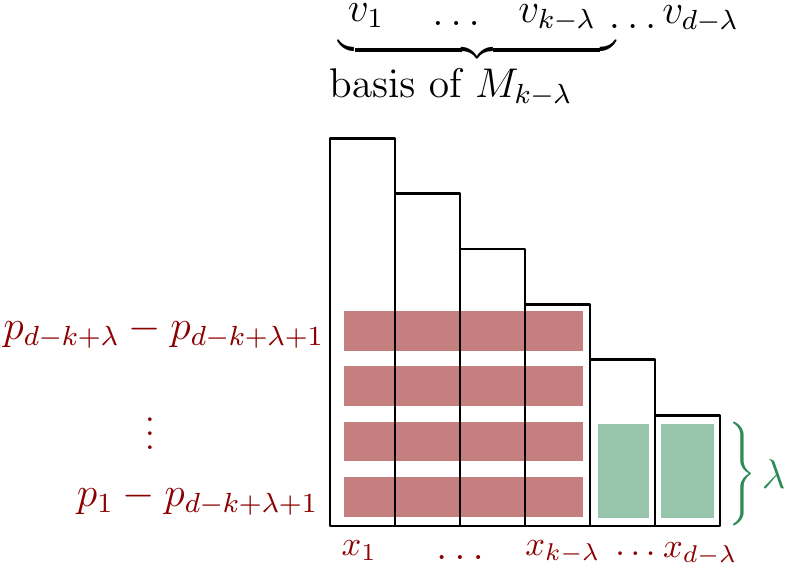}
        \caption{A visual representation of the map $f$ from the proof of \cref{thm:center-transversal-vertical}.  The first block, of width $k-\lambda$, guarantees that the points $p_1,\ldots, p_{d-k\lambda+1}$.  The second block, of height $\lambda$, guarantees that $S_{k}$ is $\lambda$-vertical.\label{fig:centertransversalmap} } 
        \end{center}
        \end{figure}
	
  The same inductive argument as in the proof of Theorem \ref{thm:rotation-case} shows that if each of $x_{k-\lambda+1}, \ldots, x_{d-\lambda}$ are zero, then each of $S_{d-1}, \ldots, S_{k}$ are $\lambda$-vertical. We do not need to use masses centered at the origin since we are only working with linear subspaces in this section.  Finally, changing the sign of $v_{k-\lambda+i}$ for $i=1,\ldots, d-k$ does not change any of the space $S_k, \ldots, S_{d-1}$ and only flips the name of $S_{d-i}^+$ and $S_{d-i}^-$.  Therefore, it only changes the sign of $x_{k-\lambda+i}$.
	
	This means that the function $f:V_{d-\lambda}(\rr^d) \to \rr^{d-1} \times \rr^{d-2} \times \ldots \times \rr^{\lambda}$ is $(\zz_2)^{d-\lambda}$-equivariant and continuous.  By \cref{thm:topo}, it must have a zero.  By construction, a zero of this function corresponds to a $\lambda$-vertical $k$-dimensional subspace $S_k$ where we can find a common $\lambda$-transversal to each mass assignment.
	\end{proof}

\section{Horizontal separating spaces}\label{sec:horizontal-separating-spaces}

A significant motivation for the study of mass partitions for mass assignments comes from the problem of splitting families of lines in $\rr^3$.  In this section, we bisect families of lines in $\rr^d$ with additional conditions.  Let $e_d \in \rr^d$ be the vertical direction.  We say that a hyperplane is vertical if it contains the direction $e_d$, and we say that an affine $(d-2)$-dimensional subspace is horizontal if it is orthogonal to $e_d$.

A family of lines in $\rr^d$ induces a mass assignment on hyperplanes in $\rr^d$, as a generic hyperplane intersects each line in a point, and this family of points can be considered as a measure.  We show that for any $d$ families of lines in $\rr^d$, there exists a vertical hyperplane $S_{d-1}$ and a horizontal $(d-2)$-dimensional affine subspace $S_{d-2} \subset S_{d-1}$ that splits each of the induced measures in $S_{d-1}$.  In the case $d=3$, this recovers the splitting result of Schnider.  However, instead of having a splitting line that intersects the $z$-axis, we get a splitting line that is perpendicular to it.

To avoid delicate issues with the continuity of the maps we define and the mass assignments, we work with measures on $A_1(\rr^d)$, the space of $1$-dimensional affine spaces of $\rr^d$.  We can consider the natural embedding $A_1(\rr^d) \hookrightarrow G_2(\rr^{d+1})$ into the Grassmannian of $2$-dimensional linear subspaces of $\rr^{d+1}$.  The manifold $G_2(\rr^{d+1})$, as a quotient space of the orthogonal group $O(d+1)\cong V_{d+1}(\rr^{d+1})$, inherits a canonical measure from the Haar measure of $O(d+1)$, which in turn induces a canonical measure in $A_1(\rr^d)$.  For simplicity, we call this the Haar measure of $A_1(\rr^d)$.  We say that a measure in $A_1(\rr^d)$ is absolutely continuous if it is absolutely continuous with respect to the Haar measure we just described.  Among other properties, this measure $\mu^*$ satisfies that for any nonzero vector $v \in \rr^d$ we have
\[
\mu^*(\{\ell \in A_1(\rr^d): \ell \perp v\})=0.
\]

Given an absolutely continuous measure $\mu$ in $A_1(\rr^d)$, this induces a mass assignment on the hyperplanes of $\rr^d$.  Given a hyperplane $H \subset \rr^d$ and a measurable set $C \subset H$, we can define
\[
\mu^H (C) = \mu \{\ell \in A_1(\rr^d) : \ell \cap H \mbox{ is a single point}, \ \ell \cap H \in C\}.
\]
In our proofs, we only evaluate these measures in closed half-spaces instead of general measurable sets $C$.  We introduce some notation to analyze the behavior of the mass assignments as our vertical hyperplanes move away from the origin. 

Given a unit vector $v \in \rr^d$ orthogonal to $e_d$ and a constant $\lambda$, consider the vertical hyperplane
\[
H_{v,\lambda} = \{p : \langle p, v \rangle = \lambda\}.
\]
Let $\ell$ be a line in $\rr^d$, not orthogonal to $v$.  If we fix $v$ and move $\lambda$ at a constant speed, the point $H_{v,\lambda} \cap \ell$ is moving at a constant speed in $H_{v,\lambda}$.  Formally, if $n(\ell)$ is a unit vector in the direction of $\ell$, then $H_{v,\lambda}$ is moving with the direction vector
\[
\frac{1}{\langle v, n(\ell)\rangle} \operatorname{proj}_{H_{v,0}}(n(\ell)).
\]

  The expression above does not depend on the choice for $n(\ell)$.  Let $z(v,\ell)$ be the $e_{d}$-component of this direction vector, which we call the vertical speed of $\ell$ in direction $v$.  For an absolutely continuous measure in $A_1(\rr^d)$ and a direction $v \perp e_d$, the set of lines with some fixed vertical speed in direction $v$ has measure zero.  Given an finite absolutely continuous measure $\mu$ in $A_1(\rr^d)$, we say that its median vertical speed for $v \perp e_d$ is the value $m = m(\mu, v)$ for which
  \[
  \mu(\{\ell \in A_1(\rr^d) : z(v,\ell) \ge m\}) = \mu(\{\ell \in A_1(\rr^d) : z(v,\ell) \le m\}).
  \]
  
  If there is a range of values that satisfy this property, we select the midpoint of this interval.  Intuitively speaking, for a finite set of lines in $\rr^d$, this parameter lets us know which line will eventually become the vertical median in $H_{v,\lambda}$ as $\lambda \to \infty$.  The advantage of working with absolutely continuous measures in $A_1(\rr^d)$ is that this parameter depends continuously on $v$.  We say that two absolutely continuous measure $\mu_1, \mu_2$ in $A_1(\rr^d)$ share a median speed in direction $v$ if $m(\mu_1, v) = m(\mu_2, v)$.  Also, we have $m(v,\mu_1) = -m(-v,\mu_1)$.  Now we are ready to state our main theorem for this section.

\begin{theorem}\label{thm:horizontal-and-lines}
	Let $d$ be a positive integer and $\mu_1, \ldots, \mu_d$ be $d$ finite absolutely continuous measures in $A_1(\rr^d)$.  Then, there either exists a direction $v \perp e_d$ in which all $d$ measures share a median speed, or there exists a vertical hyperplane $S_{d-1}$ and horizontal affine subspace $S_{d-2} \subset S_{d-1}$ such that $S_{d-2}$ bisects each of $\mu_1^{S_{d-1}}, \ldots, \mu_{d-2}^{S_{d-1}}$.
\end{theorem}

The first case, in which all measures share their median vertical speed in some direction $v$ can be interpreted as a bisection in the limit of all measures as $S_{d-1} = H_{v,\lambda}$ and $\lambda \to \infty$.

\begin{proof}
	Consider the space $S^{d-2} \times [0,1)$.  We identify the $(d-2)$-dimensional sphere $S^{d-2}$ with the set of unit vectors in $\rr^d$ orthogonal to $e_d$.  For $(v,\tau) \in S^{d-2} \times [0,1)$ define $\lambda = \tau / (1-\tau)$ and the vertical hyperplane $S_{d-1}=H_{v,\lambda}$.  Let $S_{d-2}$ be the horizontal affine subspace of $S_{d-1}$ that halves $\mu^{S_{d-1}}_{d}$.  If there is an interval of $e_d$-coordinates to choose $S_{d-2}$ from, we select the median.  The $S_{d-2}^+$ be the set of points in $S_{d-1}$ above or on $S_{d-2}$ and $S_{d-2}^-$ be the set of points in $S_{d-1}$ below or on $S_{d-2}$.
	
	We now construct the function
	\begin{align*}
	f:S^{d-2} \times [0,1) & \to \rr^{d-1} \\
		(v, \tau) & \mapsto \Big(\mu_1^{S_{d-1}}(S_{d-2}^+) - \mu_1^{S_{d-1}}(S_{d-2}^-), \ldots, \mu_{d-1}^{S_{d-1}}(S_{d-2}^+) - \mu_{d-1}^{S_{d-1}}(S_{d-2}^-)\Big)
	\end{align*}
	
	Let us analyze this function as $\tau \to 1$, which is equivalent to $\lambda \to +\infty$.  If we consider $\lambda$ to be increasing at a constant rate, then $S_{d-2}$ is moving vertically at a velocity of $m(v, \mu_2)$.  Therefore, the $i$-th component of $f$ converges to
	\[
	\mu_i (\{\ell \in A_1(\rr^d): z(v,\ell) \ge m(v, \mu_d)\}-\mu_i (\{\ell \in A_1(\rr^d): z(v,\ell) \le m(v, \mu_d)\}.
	\]
	This quantity is zero if and only if $\mu_i$ and $\mu_d$ share a median speed in direction $v$ and it depends continuously on $v$.  Therefore, we can extend this $f$ to a continuous function $\tilde{f}:S^{d-2}\times [0,1]\to \rr^{d-1}$.  If $\tilde{f} (v,\tau) = 0$ for some $(v,\tau) \in S^{d-2} \times [0,1]$, we have two cases.  If $\tau <1$, then the corresponding $S_{d-2}$ halves each of the measures $\mu_1^{S_{d-1}}, \ldots, \mu_d^{S_{d-1}}$ and we are done.  If $\tau = 1$, then all of $\mu_1, \ldots, \mu_d$ share a median speed in direction $v$, and we are done.
	
	Assume for the sake of a contradiction that $\tilde{f}$ has no zero.  We can reduce the dimension of the image by considering
	\begin{align*}
		g: S^{d-2}\times [0,1] & \to S^{d-2} \\
		(v,\tau) & \mapsto \frac{1}{\|\tilde{f}(v,\tau)\|}\tilde{f}(v,\tau),
	\end{align*}
	which is a continuous function.
	
	For each $\tau \in [0,1]$ we define the map $g_{\tau}: S^{d-2} \to S^{d-2}$ where $g_{\tau}(v) = g(v,\tau)$.  By construction $H_{v,0} = H_{-v,0}$, so $g_0(v) = g_0(-v)$.  In the other extreme, we have the function $g_1: S^{d-2} \to S^{d-2}$.  As $z(v,\ell) = -z(-v,\ell)$ and $m(\mu_d,v) = -m(\mu_d, -v)$ for all $v \in S^{d-2}, \ell \in A_1(\rr^d)$, we have $g_1(v) = -g_1(-v)$ (intuitively, as we change from $v$ to $-v$ everything but the vertical direction flips, so the sign of $g_1$ changes).  Therefore, $g$ would be a homotopy between an even map of $S^{d-2}$ and an odd map of $S^{d-2}$, which is impossible.
	
\end{proof}

\begin{figure}[h!]
       \begin{center}
       \includegraphics[width=0.9\textwidth]
        {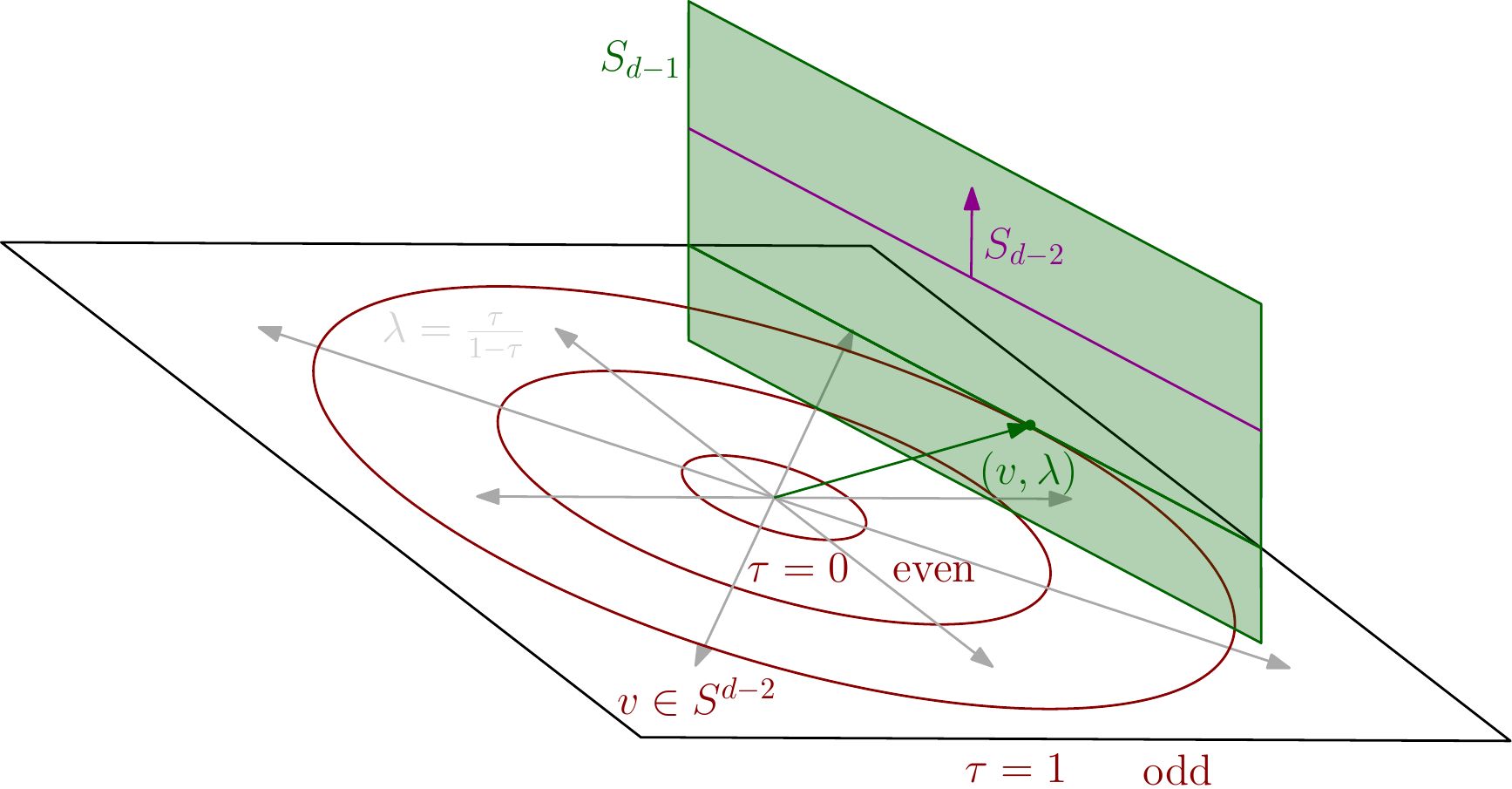}
        \caption{Once we choose $(v,\lambda)$ the hyperplane $S_{d-1}$ is determined.  We then choose $S_{d-2} \subset S_{d-1}$ at the appropriate height so it halves $\mu^{S_{d-1}}_{d}$.\label{fig:horizontal} } 
        \end{center}
        \end{figure}

\section{Dynamic ham sandwich theorems}\label{sec:dynamic-ham-sandwich}

In this section, we are concerned by mass assignments induced by families of affine $(d-k)$-spaces in $\rr^d$, as we seek a translate $S_k$ of a particular space of dimension $k$ in which we can split the measures induced on $S_k$.  We aim to bisect simultaneously $d$ mass assignments.  Let us start with $k=d-1$.  As described in the introduction, this problem can be considered as a dynamic ham sandwich problem in $\rr^d$.

We first describe a continuous version of this theorem.  We increase the dimension by one for convenience.  Consider $\rr^d$ embedded in $\rr^{d+1}$ as the set of points $x$ for which the last coordinate is zero, so $\langle x, e_{d+1}\rangle = 0$.  Given an absolutely continuous measure $\mu_i$ on $A_1(\rr^{d+1})$ we denote my $\mu_{i,t}$ the measure it induces on the hyperplane defined by $\langle x, e_{d+1}\rangle=t$.  We refer to this as a \textit{moving measure} in $\rr^d$, where we think of the variable $t$ as the time.  Each $\ell \in A_1(\rr^{d+1})$ not parallel to $e_{d+1}$ corresponds to a pair $(p,w) \in \rr^d \times \rr^d$.  At time $t$ it is represented by the point $p+tw$.  Therefore, a moving measure can be thought of as a measure on $\rr^d \times \rr^d$.   The extension to absolutely continuous moving measures is to make the topological maps we construct continuous.  If $v$ is a unit vector in $\rr^d$, we say that the speed of $(p,w)$ in the direction of $v$ is $\langle v, w\rangle$.

Given a moving measure $\mu$ in $\rr^d$ and a unit vector $v \in \rr^d$, we call the \textit{median speed} of $\mu$ in the direction $v$ as the value $m$ for which

\[
\mu\{(p,w) \in \rr^d \times \rr^d: \langle w , v \rangle \le m \} = \mu\{(p,w) \in \rr^d \times \rr^d: \langle w , v \rangle = m \}
\]

\begin{theorem}\label{thm:translation-strange}
	Let $d$ be a positive odd integer.  Let $\mu_0, \mu_1, \ldots, \mu_{d}$ be $d+1$ finite absolutely continuous moving measures on $\rr^d$.  Then, there either exists a direction $v$ in which all measures have the same median speed or there exists a time $t\in(-\infty,\infty)$ and a hyperplane $S_{d-1} \subset \rr^d$ so that its closed half-spaces $S^+_{d-1}, S^-_{d-1}$ satisfy
	\[
	\mu_{i,t} (S^+_{d-1}) = \mu_{i,t} (S^-_{d-1}) \qquad \mbox{ for each }i=0,1,\ldots, d.
	\]
\end{theorem}

\begin{proof}
	  We parametrize our partitions with pairs $(v,\tau) \in S^{d-1}\times (-1, 1)$.  Let $S_{d-1}$ be the hyperplane orthogonal to $v$ that halves $\mu_{0,t}$ for $t = \tau / (1-|\tau|)$.  As usual, if there is a range of possibilities for $S_{d-1}$ we choose the median.  We denote by $S^{+}_{d-1}$ the closed half-space in the direction of $v$ and $S^{-}_{d-1}$ the closed half-space in the direction of $-v$.
	
	Now we can construct a function
	\begin{align*}
		f: S^{d-1}\times (-1, 1) & \to \rr^d \\
		(v,\tau) & \mapsto \Big(\mu_{1,t}(S^{+}_{d-1}) - \mu_{1,t}(S^{-}_{d-1}), \ldots, \mu_{d,t}(S^{+}_{d-1}) - \mu_{d,t}(S^{-}_{d-1})\Big)
	\end{align*}
	
	This function is continuous.  If $v$ is fixed and $\tau \to 1$, then $f(v,\tau)$ converges.  For example, the first coordinate converges to $\mu_1(A)-\mu_1(B)$, where $A$ is the set of moving points whose speed in the direction $v$ is larger than or equal to the median speed in the direction $v$ for $\mu_0$, and $B$ is the complement of $A$.  The same convergence holds when $\tau \to -1$.  Therefore, we can extend $f$ to a continuous function $\tilde{f}: S^{d-1} \times [-1,1] \to \rr^d$.
	
	If we show that $\tilde{f}$ has a zero, we have our desired conclusion.  Since the values at of $\tilde{f}(v,1)$ and $\tilde{f}(v,-1)$ only depend on the median speeds described above, for all $v\in S^{d-1}$ we have $f(v,-1) = -f(v,1)$.  Also, by the definition of $f$ we know that if we fix $\tau$ and flip $v$ to $-v$, the hyperplane $S_{d-1}$ does not change but $S^{+}_{d-1}$ and $S^{-1}_{d-1}$ swap.  Therefore, for every $\tau$ we have $\tilde{f}(v,\tau)=-\tilde{f}(-v,\tau)$.
	
	Assume $\tilde{f}$ does not have a zero, and we look for a contradiction.  We can reduce the dimension of the image by considering
	\begin{align*}
		g: S^{d-1}\times[-1,1] & \to S^{d-1} \\
		(v,\tau) & \mapsto \frac{1}{\|\tilde{f}(v,\tau)\|}\tilde{f}(v,\tau)
	\end{align*}
	
	As before, $g(v,-1) = -g(v,1)$ for every $v \in S^{d-1}$ and $g(v,\tau)=-g(-v,\tau)$ for every $(v,\tau) \in S^{d-1}\times [-1,1]$.  Additionally, $g$ is continuous.  If we fix $\tau$, we get continuous functions on the sphere.
	
\begin{align*}
	g_{\tau}: S^{d-1} & \to S^{d-1} \\
	v & \mapsto g(v,\tau).
\end{align*}

Let us analyze the degree $g_{-1}$ and $g_1$.  As $g_{\tau}(v) = -g_{\tau}(-v)$, both maps are of odd degree.  As $g_{-1}(v) = -g_1(v)$, we have $\deg (g_{-1}) = (-1)^d \deg (g_1) \neq \deg (g_1)$, since $d$ is odd and $\deg (g_1) \neq 0$.  Therefore, it is impossible to have a homotopy between these two maps, which is the contradiction we wanted.
\end{proof}

The condition on the parity of $d$ may seem like an artifact of the proof, but we can see in \cref{fig:dynamic-example} that three moving measures in $\rr^2$ concentrated around the moving points depicted will not have a common halving line at any point in time.

The example can be extended to high dimensions.  We embed $\rr^d$ into $\rr^{d+1}$ as as the set of points whose coordinates add to $1$.  We can consider $d+1$ moving measures in $\rr^{d}$ by taking the moving pairs $(e_i, e_{i+1}-e{i}) \in \rr^{d+1} \times \rr^{d+1}$, where we consider $e_{d+2} = e_1$.  Note that these points always lie on the embedding of $\rr^d$ in $\rr^{d+1}$.  To show that there is never a hyperplane of $\rr^d$ going through all the points, we have to check that, for even $d$, the $(d+1) \times (d+1)$ matrix
\[
A=\begin{bmatrix}
1-t & 0 & 0 &\ldots & 0 & t \\
t & 1-t & 0 &\ldots & 0 & 0 \\
0 & t & 1-t & \ldots & 0 & 0 \\
\vdots & \vdots & \vdots & \ddots & \vdots & \vdots \\
0 & 0 & 0 & \ldots & 1-t & 0 \\
0 & 0 & 0 & \ldots & t & 1-t
\end{bmatrix}.
\]
is non-singular for all values of $t$.  A simple pattern chasing or row expansion shows that $\det (A) = (1-t)^{d+1}+(-1)^d t^{d+1}$.  For $t\in [0,1]$ and even $d$, $\det(A) > 0$.  For $t \not\in [0,1]$, we have $|1-t| \neq |t|$, so $\det(A) \neq 0$.

\begin{figure}
	\centerline{\includegraphics{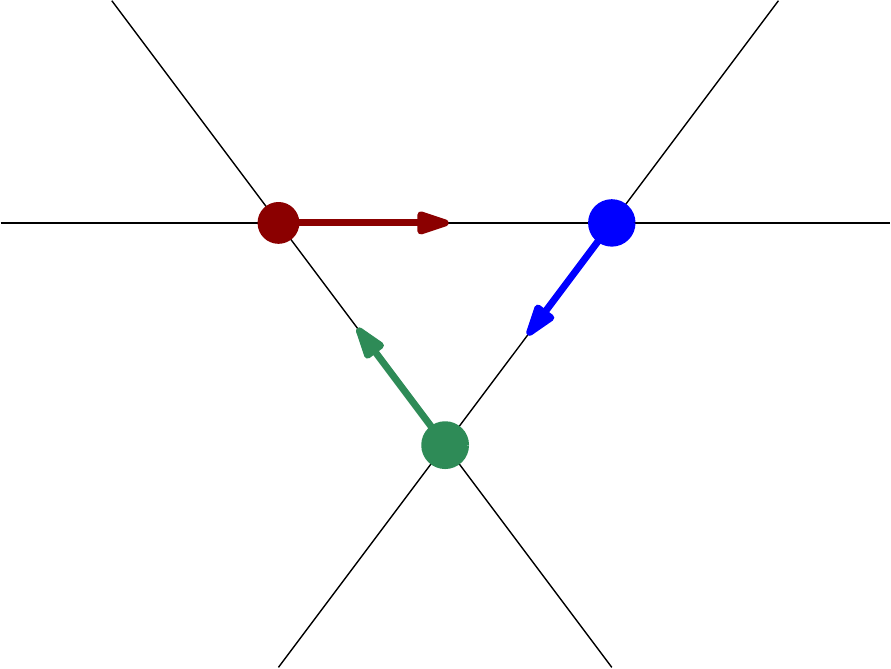}}
	\caption{Three moving points in the plane such that at no point in time there is a line that simultaneously goes through all three.}
	\label{fig:dynamic-example}
\end{figure}

Now we work with the case $k=1$ for the problem stated at the start of this section.  Given $d$ families of hyperplanes in $\rr^d$, we look for a translate of a particular line in which we can find a point that halves all of them.  For this problem, the parity of $d$ is not important.

We assume without loss of generality that the line $\ell$ we will translate is $\operatorname{span}(e_{d})$.  Assume $H$ is a hyperplane with a normal vector $n$ so that $\langle n, e_{d} \rangle \neq 0$.  This guarantees that $H$ intersects every translation of $\ell$ at a single point $p$.  If $v$ is a unit vector orthogonal to $e_{d}$, as we translate $\ell$ in the direction $v$, the point $\ell \cap H$ moves vertically at a speed $m(H,v)=- \langle n, v \rangle / \langle n , e_{d} \rangle$.  As expected, the vertical speeds in direction $v$ and in direction $-v$ are negative of each other.

As in \cref{sec:horizontal-separating-spaces}, we can define an absolutely continuous measure on $A_{d-1}(\rr^d)$ using the Haar measure on $G_d(\rr^{d+1})$.  This case is perhaps more intuitive, as $G_d(\rr^{d+1})$ is a $d$-dimensional projective space.  The canonical measure on this projective space comes from the Haar measure on its double cover, $S^d$.   If $\mu$ is an absolutely continuous measure on $A_{d-1}(\rr^{d})$, the set of hyperplanes whose normal vectors are orthogonal to $e_{d}$ has measure zero.

Given an absolutely continuous measure $\mu$ in $A_{d-1}(\rr^d)$ and a line $\ell$ in $\rr^d$, we can define a measure $\mu^{\ell}$ on $\ell$ by considering
\[
\mu^{\ell}(C) = \mu\{H \in A_{d-1}(\rr^d): H \cap \ell \mbox{ is a point and }H \cap \ell \in C\}.
\]
We define the \textit{median speed} of $\mu$ in the direction $v$ as the number $m$ such that
\[
\mu\{H \in A_{d-1}(\rr^d) : m(H,v) \le m\} = \mu\{H \in A_{d-1}(\rr^d) : m(H,v) \ge m\}.
\]

\begin{theorem}\label{thm:translated-line}
	Let $d$ be a positive integer.  Let $\mu_1, \ldots, \mu_{d}$ be finite absolutely continuous measures on $A_{d-1}(\rr^d)$.  Then, there exists a either exist a direction $v \perp e_d$ in which all measures have the same median vertical speed or there translate $\ell$ of $\operatorname{span}(e_{d})$ and a point that splits $\ell$ into two complementary closed half-lines $A$, $B$ such that
	
	\[
	\mu^\ell_{i} (A) = \mu^{\ell}_{i} (B) \qquad \mbox{ for each }i=0,1,\ldots, d.
	\]
\end{theorem}

\begin{proof}
Our configuration space this time is $S^{d-2} \times [0,1)$.  We consider $S^{d-2}$ the set of unit vectors in $\rr^d$ orthogonal to $e_{d}$.  For $(v,\tau) \in S^{d-2} \times [0,1)$, let $\ell$ be the translate of $\operatorname{span}(e_d)$ by the vector $(\tau/(1-\tau)) v$.  We consider $h$ to be the height of the median of $\mu_d^{\ell}$ and $A$ the set of points on $\ell$ at height $h$ or more, and $B$ the set of points of $\ell$ at height $h$ or less.

Now we define
\begin{align*}
    f: S^{d-2} \times [0,1) & \to \rr^{d-1} \\
    (v,\tau) & \mapsto (\mu_1^\ell (A) - \mu_1^\ell (B),\ldots, \mu_{d-1}^\ell (A) - \mu_{d-1}^\ell (B))
\end{align*}

For a fixed $v$, the point $f(v,\tau)$ converges as $\tau \to 1$.  This is precisely because $\mu_1^{\ell}(A)$ converges to the measure $\mu_1$ of the subset of hyperplanes in $A_{d-1}(\rr^d)$ whose speed in the direction $v$ is at least the median vertical speed of $\mu_d$ in the direction $v$.  Therefore, we can extend $f$ to a continuous function $\tilde{f}: S^{d-2} \times [0,1] \to \rr^{d-1}$.  If $\tilde{f}$ has a zero, we obtain the desired conclusion.  Otherwise, we can reduce the dimension of the image and construct our final continuous map
\begin{align*}
    g: S^{d-2} \times [0,1] & \to S^{d-2} \\
    (v,\tau) & \mapsto \frac{1}{\|\tilde{f}(v,\tau)\|}f(v,\tau)
\end{align*}

The map $g_1: S^{d-2} \to S^{d-2}$, defined by $g_1(v) = g(v,1)$, is odd.  This is because the speed of a hyperplane in direction $v$ is the negative of its speed in direction $-v$.  Therefore, the degree of $g_1$ is odd.  On the other hand, $g_0: S^{d-2} \to S^{d-2}$, defined by $g_0(v) = g(v,0)$, is a constant map, so its degree is zero.  There cannot be a homotopy between $g_1$ and $g_0$.
\end{proof}

% \bib, bibdiv, biblist are defined by the amsrefs package.

\end{document}